\documentclass[]{amsart}
\usepackage{amsmath, amsfonts, amssymb, amscd, amsthm}
\usepackage[mathcal]{euscript}
\usepackage{stmaryrd}
\usepackage{graphicx}
\usepackage{epstopdf}

\theoremstyle{plain}

\newtheorem{theorem}{Theorem}[section]
\newtheorem{proposition}[theorem]{Proposition}
\newtheorem{lemma}[theorem]{Lemma}

\theoremstyle{definition}
\newtheorem{definition}[theorem]{Definition}
\theoremstyle{remark}
\newtheorem*{remark}{Remark}

\numberwithin{equation}{section}

\newcommand{\sthree}{{S}^3}
\newcommand{\rfour}{\mathbb{R}^4}
\newcommand{\Ker}{\operatorname{Ker}}

\newcommand{\slk}{\operatorname{slk}}
\newcommand{\lsh}{\textsc{ls}-homotopy}
\newcommand{\lshc}{\textsc{ls}-homotopic}

\newcommand{\sone}{{S}^1}

\newcommand{\rthree}{\mathbb{R}^3}
\newcommand{\rtwo}{\mathbb{R}^2}

\newcommand{\rfive}{\mathbb{R}^5}
\newcommand{\sfour}{S^4}
\newcommand{\R}{\mathbb{R}}
\newcommand{\Z}{\mathbb{Z}}
\newcommand{\rank}{\operatorname{rank}}
\newcommand{\hv}{\llbracket}
\newcommand{\hh}{\rrbracket}

\begin{document}
\title[Locally stable maps $\sthree\to\rfour$]{A geometric classification of the path components of the space of locally stable maps $\sthree\to\rfour$}
\author{Ole Andersson}
\thanks{I would like to thank my adviser Tobias Ekholm for suggesting 
the problem studied in this paper and for many 
fruitful discussions.}
\address{Department of Mathematics, Uppsala University, SE-751 06 Uppsala, Sweden}
\email{Ole.Andersson@math.uu.se}

\begin{abstract}
Locally stable maps $\sthree\to\rfour$ are classified up to
homotopy through locally stable maps. The equivalence class of a
map $f$ is determined by three invariants: the isotopy class
$\sigma(f)$ of its framed singularity link, the generalized normal
degree $\nu(f)$, and the algebraic number of cusps $\kappa(f)$ of
any extension of $f$ to a locally stable map of the $4$-disk into
$\rfive$. Relations between the invariants are described, and it
is proved that for any $\sigma$, $\nu$, and $\kappa$ which satisfy
these relations, there exists a map $f:\sthree\to\rfour$ with
$\sigma(f)=\sigma$, $\nu(f)=\nu$, and $\kappa(f)=\kappa$. It
follows in particular that every framed link in $\sthree$ is the
singularity set of some locally stable map into $\rfour$.
\end{abstract}

\maketitle
\section{Introduction}
\label{introduction}
This paper concerns locally stable codimension one maps from 
spheres into Euclidean spaces. Locally stable maps are maps with
stable map germs, see Section \ref{background}. In the nice
dimensions of Mather \cite{Mat}, in particular in the case of
codimension one maps of spheres of dimensions less than $15$,
locally stable maps constitute an open and dense subset of
$C^\infty(S^k,\R^n)$, the space of smooth maps $S^k\to\R^n$.
A homotopy through locally stable maps will be called an
\lsh.

Every locally stable map $S^1\to\R^2$ is an immersion, and the
classification of such maps up to \lsh\ reduces to
the well-known classification of immersed plane curves up to
regular homotopy, see Whitney \cite{Whi}. In the case of maps
$S^2\to\R^3$, a locally stable map need not be an immersion. It
may have a finite number of singularities, where the rank of the
differential equals $1$ (so called Whitney umbrellas). Juh\'asz
\cite{Juh} proved that two singular locally stable maps from a closed
connected surface into $\rthree$ are \lshc\ if and only if they
have the same number of singularities. The results of \cite{Juh} have 
been generalized by Juh\'asz \cite{Juh1} to locally stable maps from any closed
smooth $n$-manifold to $\R^{2n-1}$ when $n\ne 3$.

In the present paper we classify locally stable maps $S^3\to\R^4$
up to \lsh. The classification is more refined than the
corresponding classification for maps $S^k\to\R^{k+1}$ when $k\leq
2$. It involves an extension of geometrically defined regular
homotopy invariants of immersions $\sthree\to\rfour$ as well as
link theory in $\sthree$.

A locally stable map $f\colon S^3\to\R^4$ is an immersion outside
a closed $1$-dimensional submanifold $\Sigma(f)$ of $S^3$. Along
$\Sigma(f)$ the rank of the differential $df$ of $f$ equals $2$.
(In fact the restriction of $f$ to a neighborhood of $\Sigma(f)$
can be described as a $1$-parameter family of Whitney umbrellas,
see Section \ref{background}.) The kernel bundle of $df$ is a
trivial line bundle along $\Sigma(f)$ which is nowhere tangential
to $\Sigma(f)$, see Proposition \ref{fredrik}. Hence $\Sigma(f)$
is a framed link in $\sthree$.

Using an idea inspired by a construction in Goryunov \cite{Gor},
we extend the definition of the normal degree for immersions
$S^3\to\R^4$ so that it applies also to locally stable maps
$S^3\to\R^4$, see Definition \ref{def_of_nu}. Following Ekholm and
Sz{\H u}cs \cite{EkSz}, we consider the algebraic number of cusps
($\Sigma^{1,1}$-points) of a locally stable map $(D^4,\partial
D^4)\to(\R^5,\R^4)$ which extends a given locally stable map
$\partial D^4=S^3\to\R^4$ and which agrees with the product
extension in some collar neighborhood of $S^3$ in $D^4$, see
Definition \ref{def_of_kappa}.

To state the classification theorem we introduce the following
notation. For a locally stable map $f\colon S^3\to\R^4$ let
$\sigma(f)$ be the isotopy class of the framed singularity link
$\Sigma(f)$ of $f$ (if $f$ is an immersion let $\sigma(f)$ be the
class of the empty link), let $\nu(f)$ be the (generalized) normal
degree of $f$, and let $\kappa(f)$ be the algebraic number of
cusps of an extension of $f$ to the disk as discussed above.

\begin{theorem}
\label{ett}
Two locally stable maps $f,g:\sthree \to \rfour$ are homotopic
through locally stable maps if and only if $\sigma(f)=\sigma(g)$,
$\nu(f)=\nu(g)$, and $\kappa(f)=\kappa(g)$.
\end{theorem}

It follows from Theorem \ref{ett} that two immersions
$f,g:\sthree\to\rfour$ are regularly homotopic if and only if
$\nu(f)=\nu(g)$ and $\kappa(f)=\kappa(g)$. This result can also be
derived from \cite{Ekh2} or \cite{Hug} together with \cite{EkSz}.
When restricted to immersions, the invariant $\kappa$ takes values
in $2\Z$, see \cite{EkSz}. Furthermore, if $f:\sthree\to\rfour$ is
an immersion, then $\nu(f)$ and $\frac{1}{2}\kappa(f)$ have
different parity, see Proposition \ref{immpar}.

\begin{theorem}
\label{tva}
Let $\sigma$ be an isotopy class of framed links in $S^3$. Then
there exists a locally stable map $f:S^3\to\R^4$ such that
$\sigma(f)=\sigma$. Moreover, for any locally stable map
$g:S^3\to\R^4$ with $\sigma(g)=\sigma$ there exists an immersion
$h:\sthree\to\rfour$ such that $g$ is homotopic through locally
stable maps to the connected sum $f\natural h$. Hence there are
integers $a$ and $b$ such that
\begin{equation}
\label{eqrest}
\nu(g)=\nu(f)+a\quad\text{and}\quad\kappa(g)=\kappa(f)+2a+4b.
\end{equation}
Conversely, if $a$ and $b$ are any integers, then there exists a
locally stable map $g:S^3\to\R^4$ that satisfies
$\sigma(g)=\sigma$ and Equations \eqref{eqrest}.
\end{theorem}

Equations \eqref{eqrest} imply that the parity of $\kappa(g)$ and the residue
class of $2\nu(g)+\kappa(g)$ modulo 4 depend only on $\sigma$. We
show that the parity equals the sum of the modulo 2 self linking
numbers of all components of the framed link $\Sigma(g)$, see
Proposition \ref{slk}. On immersions the residue class equals 2
modulo 4, see Proposition \ref{immpar}. The behavior of the
residue class on locally stable maps with nonempty singularity
sets will be studied in a forthcoming paper.

The paper is organized as follows. Section \ref{background}
contains a short background to stable maps and singularities.
Section \ref{invariants} discusses the invariants $\nu$ and
$\kappa$. In Section \ref{realization} we prove that every framed
link in $\sthree$ can be realized as a framed singularity link of
a locally stable map into $\rfour$. Sections \ref{sufficiency} and
\ref{applications} are devoted to the proofs of Theorems \ref{ett}
and \ref{tva}. 

\section{Stable maps and singularities}
\label{background}
This section contains background material from the theory of
stable maps. In Subsections \ref{St_o_Loc} and \ref{sing} we
introduce the concepts of stability and local stability of maps,
and review the Thom-Boardman theory of singularities. For
convenience most definitions are stated in greater generality than
later needed. In Subsections \ref{ulla} and \ref{totto} we focus on
stable and locally stable maps $\sthree\to\rfour$, and give a
detailed description of their properties.
In particular we define the invariant $\sigma$
mentioned in Section \ref{introduction}. When not included, proofs
of the statements in Subsections \ref{St_o_Loc}--\ref{totto} can
be found in \cite{GoGu}. In Subsection \ref{orient} we fix
metric and orientation conventions.

\subsection{Stable and locally stable maps}
\label{St_o_Loc}
Let $M$ and $N$ be smooth manifolds, with $M$ compact, and denote
by $C^\infty(M,N)$ the space of smooth maps $M\to N$ equipped with
the Whitney $C^\infty$ topology. We assume, unless otherwise is
explicitly mentioned, all manifolds and maps to be smooth.

Recall that a map $f:M \to N$ is \emph{stable} if each map in a
neighborhood of $f$ in $C^\infty(M,N)$ is a composition $\psi
\circ f \circ \varphi$ with $\varphi:M \to M$ and $\psi:N \to N$
being diffeomorphisms. Stable maps form an open subset of
$C^\infty(M,N)$, and in the nice dimensions of Mather \cite{Mat}
the set of stable maps is dense.

Mather \cite{Mat2} proved that a map $f:M\to N$ is stable if and
only if $f$ is \emph{infinitesimally stable}, that is, if for each
vector field $w$ along $f$ there are vector fields $u$ on $M$ and
$v$ on $N$ such that
\begin{equation}
\label{eq1}
w = df\circ u + v\circ f.
\end{equation}
We call a map $f:M\to N$ \emph{locally stable} if its map germs
are infinitesimally stable. That is, if for each $p\in M$ and each
germ of a vector field $w$ along $f$ at $p$ there are germs of
vector fields $u$ on $M$ at $p$ and $v$ on $N$ at $f(p)$ such that
Equation \eqref{eq1} is satisfied in a neighborhood of $p$. By
\cite{Mat2} every stable map is locally stable.

A homotopy $h_t:M \rightarrow N$ will be called an
\emph{\textsc{ls}-homotopy} if each $h_t$ is locally stable.
{\sc{Ls}}-homotopies correspond to continuous paths in the space
of locally stable maps, and locally stable maps $f,g:M\rightarrow
N$ are \emph{\lshc}\ if there exists an \lsh\ from $f$ to $g$.

\subsection{Singularities}
\label{sing}
A point $p\in M$ is a \emph{singularity} of a map $f:M \rightarrow
N$ if the differential $df$ of $f$ has positive corank at $p$:
\begin{equation*}
\rank df(p)<\min\{\dim M,\dim N\}.
\end{equation*}
We call $f$ \emph{singular} if the set of singularities of $f$ is nonempty.

The singularities of a map $f:M\to N$ are classified
according to the corank of $df$ and the coranks of the
differentials of restrictions of $f$ to submanifolds of
singularities. For each integer $r$ let $\Sigma^r(f)$ be the set
of points in $M$ at which $df$ has corank $r$. If $\Sigma^r(f)$ is
a nonempty manifold consider the restriction $f|_{\Sigma^r(f)}$
and for each integer $s$ define $\Sigma^{r,s}(f)$ to be the set of
points in $\Sigma^r(f)$ at which the differential
$d(f|_{\Sigma^r(f)})$ has corank $s$. Analogously define for every
sequence of integers $r_1,\dots,r_k$ the set
$\Sigma^{r_1,\dots,r_k}(f)$ to be the set
$\Sigma^{r_k}(f|_{\Sigma^{r_1,\dots,r_{k-1}}(f)})$ provided that
$\Sigma^{r_1,\dots,r_{k-1}}(f)$ is defined and is a nonempty
submanifold of $M$. Boardman \cite{Boa} associated to every
sequence of integers $r_1\geq\dots\geq r_k\geq 0$ a fiber
subbundle $\Sigma^{r_1,\dots,r_k}$ of the jet bundle $J^k(M,N)$
and proved for the maps $f$ in a residual subset of
$C^\infty(M,N)$ that each jet extension $j^kf$ is transverse to
every $\Sigma^{r_1,\dots,r_k}$, and that
$\Sigma^{r_1,\dots,r_k}(f)=(j^kf)^{-1}(\Sigma^{r_1,\dots,r_k})$.

The jet extensions of a locally stable map are transverse to every
$\Sigma^{r_1,\dots,r_k}$. Therefore the set of singularities of a
locally stable map $f$ is the union of the manifolds
$\Sigma^{r_1,\dots,r_k}(f)$ with $r_1$ positive.

If $\dim M\leq \dim N$, then the restriction of a stable map $f:M\rightarrow N$ to
each nonempty manifold $\Sigma^{r_1,\dots,r_k}(f)$ with $r_k=0$ is
an immersion with normal crossings. Hence, if $\Sigma$ is
shorthand for $\Sigma^{r_1,\dots,r_k}(f)$, then for each $q\in N$
and each set of distinct points $p_1,\dots,p_n\in
f^{-1}(q)\cap\Sigma$ the linear subspaces
$df(T_{p_1}\Sigma),\dots,df(T_{p_n}\Sigma)$ of $T_qN$ are in
general position:
\begin{equation}
\label{gen_pos}
\text{codim}\big(\bigcap_{j=1}^{n}
df(T_{p_j}\Sigma)\big)=\sum_{j=1}^{n}\text{codim}\big(df(T_{p_j}\Sigma)\big).
\end{equation}
In particular, $f|_\Sigma$ has transverse self intersections, and
the cardinality of the preimage of any point in $N$ under
$f|_\Sigma$ is bounded from above by $\dim N/(\dim N - \dim
\Sigma)$.

\subsection{Stable and locally stable maps from $\sthree$ to $\rfour$}
\label{ulla}
Locally stable maps constitute an open and dense subset of
$C^\infty(\sthree,\rfour)$. The singularities of a singular
locally stable map $f:\sthree \rightarrow \rfour$ are all of type
$\Sigma^{1,0}$. The set $\Sigma^{1,0}(f)$ is a closed
1-dimensional submanifold of $\sthree$. From now on we denote
$\Sigma^{1,0}(f)$ by $\Sigma(f)$ and call $\Sigma(f)$ \emph{the
link of singularities of $f$}.

\begin{proposition}
\label{martin}
Let $h_t:\sthree\to\rfour$ be an \lsh. Then the links
$\Sigma(h_0)$ and $\Sigma(h_1)$ are isotopic.
\end{proposition}

\begin{proof}
Apply \cite[Theorem 20.2]{AbRo} to the homotopy
$j^1h_t:\sthree\rightarrow J^1(\sthree,\rfour)$ and the (closure
of) the submanifold $\Sigma^1$ of $J^1(\sthree,\rfour)$.
\end{proof}

A locally stable map $\sthree\to\rfour$ is characterized by its
normal forms in local coordinates. A map $f:\sthree \rightarrow
\rfour$ is locally stable if and only if for each $p\in\sthree$
there are charts $\alpha:U \rightarrow \mathbb{R}^3$ on $\sthree$
centered at $p$ and $\beta:V \rightarrow \rfour$ on $\rfour$
centered at $f(p)$ such that $f(U)\subset V$ and such that $\beta
\circ f \circ \alpha^{-1}:\alpha(U) \rightarrow \rfour$ has normal
form
\begin{equation*}
(\beta \circ f \circ \alpha^{-1})(x,y,z)=\left\{ \begin{array}{lll}
(x,y,z,0) & \text{if} & p\in \sthree \setminus \Sigma(f), \\
(x,y,z^2,yz) & \text{if} & p\in \Sigma(f).
\end{array} \right.
\end{equation*}

Every locally stable map $\sthree\to\rfour$ can be deformed into a
stable map using an arbitrarily small \lsh. The restrictions
$f|_{\Sigma(f)}$ and $f|_{\sthree \setminus \Sigma(f)}$ of a stable map
$f:\sthree\to\rfour$ are immersions with normal crossings, and the images of these
immersions intersect transversely. By \eqref{gen_pos}, $f|_{\Sigma(f)}$ 
is an embedding, and the preimage
$f^{-1}(q)$ of each $q\in\rfour$ contains at most 4 points. In
fact, if $f^{-1}(q)\cap \Sigma(f)\ne\emptyset$, then $f^{-1}(q)$
contains at most 2 points.

\begin{proposition}
\label{sture}
Let $f:\sthree\to\rfour$ be a singular stable map, and let $K$ be
a component of $\Sigma(f)$. Then there are tubular neighborhoods
$\alpha:\sone\times D^2_r\to U$ of $K$ and
$\beta:\sone\times D^3_s\to V$ of $f(K)$ such that $f(U)\subset
V$ and such that $\beta^{-1}\circ
f\circ\alpha:\sone\times D^2_r\to\sone\times D^3_s$ is given by
\begin{equation}
\label{normform}
(\beta^{-1}\circ f\circ\alpha)(\theta,x,y)=(\theta,x,y^2,xy).
\end{equation}
\end{proposition}

\begin{proof} This proposition is a special case of a general theorem of Sz\H{u}cs \cite{Szu1}.
Note that in our case the normal bundles of $\Sigma(f)$ and $f(\Sigma(f))$ are trivial. 
\end{proof}

\subsection{The kernel bundle and framed links in $\sthree$.}
\label{totto}
Let $f:\sthree\rightarrow\rfour$ be a singular locally stable map.
The differential of $f$ restricts to a linear corank one map
$df:T\sthree|\Sigma(f)\to T\rfour$. Let $\Ker(df)$ be its kernel
line bundle. Then $\Ker(df)$ is nowhere tangential to $\Sigma(f)$.

\begin{proposition}
\label{fredrik}
If $f:\sthree \rightarrow \rfour$ is a singular locally stable
map, then $\Ker(df)$ is trivial.
\end{proposition}

\begin{proof}
Assume first that $f$ is stable. Let $U$ be an open set in
$\sthree$ that is properly contained in
$\sthree\setminus\Sigma(f)$ and which covers the finitely many
points in $\sthree\setminus\Sigma(f)$ with image in
$f(\Sigma(f))$. Let $\bar{U}$ be the closure of $U$ in $\sthree$,
and let $V$ be the set remaining when all triple lines and
quadruple points for the restriction $f|_{\sthree\setminus\bar U}$ are
removed from $\sthree\setminus\bar U$. Then $V$ is an open
neighborhood of $\Sigma(f)$ such that $|f^{-1}(f(p))\cap V|\leq 2$
for all $p\in V$ and $f^{-1}(f(p))\cap V= \{p\}$ for all
$p\in\Sigma(f)$.

The restriction $f|_{V\setminus\Sigma(f)}$ is an immersion with
normal crossings and at most double points. Hence $\hat S=\{p\in
V:|f^{-1}(f(p))\cap V|=2\}$ and $S=f(\hat S)$ are surfaces in
$\sthree$ and $\rfour$ respectively, and $f|_{\hat S}:\hat
S\rightarrow S$ is a double covering. By \cite[Proposition 2.5]{Ekh1},
which applies to arbitrary connected manifolds, 
since $\dim \rfour - \dim \sthree$ is odd, the
covering $f|_{\hat S}$ is the orientation double covering of $S$.
Hence $\hat S$ is naturally oriented.

Let $Y=\hat S\cup\Sigma(f)$. Then $Y$ is a surface in $\sthree$,
and the orientation of $\hat S$ extends to an orientation of $Y$.
Moreover, the induced bundle $TY|\Sigma(f)$ equals $\Ker(df)\oplus
T\Sigma(f)$. Because $T\Sigma(f)$ is orientable so is $\Ker(df)$.
Consequently $\Ker(df)$ is trivial.

Now assume $f$ is merely locally stable. Let
$h_t:\sthree\to\rfour$ be an \lsh\ from a stable map $g$ to $f$,
and let $J:\sthree\times I\to J^1(\sthree,\rfour)$ be the homotopy
$J(p,t)=j^1h_t(p)$. Then $J$ is transverse to $\Sigma^1$ and, by
\cite[Theorem 20.2]{AbRo}, the manifold
$\Omega=J^{-1}(\Sigma^1)$ is a product cobordism from
$\Sigma(g)\times \{0\}$ to $\Sigma(f)\times\{1\}$.

Let $H:\sthree\times I\to\rfour\times I$ be the track of the
homotopy $h_t$. The differential of $H$ restricts to a linear
corank one map $dH:T(\sthree\times I)|\Omega\to T(\rfour\times
I)$. Let $\Ker(dH)$ be its kernel line bundle. Then $\Ker(dH)$ is
trivial since $\Ker(dH)|(\Sigma(g)\times\{0\})$ is isomorphic to
the trivial bundle $\Ker(dg)$. Hence also $\Ker(df)$ is trivial
since $\Ker(df)$ is isomorphic to $\Ker(dH)|(\Sigma(f)\times
\{1\})$.
\end{proof}

A \emph{framed link} in $\sthree$ is a pair $(L,u)$ where $L$ is a
link in $\sthree$ and $u:L\to T\sthree|L$ is a section such that
$u(p)$ is non-tangential to $L$ for every $p\in L$. An
\emph{isotopy of framed links in $\sthree$} is a 1-parameter
family of framed links $(L_t,u_t)$ such that $L_t=\varphi_t(L_0)$
for an isotopy $\varphi_t:L_0\to\sthree$, and
$u_t\circ\varphi_t:L_0\to T\sthree$ is a homotopy. The framed
links $(L_0,u_0)$ and $(L_1,u_1)$ are then called \emph{isotopic},
and the family $(L_t,u_t)$ is called an \emph{isotopy from
$(L_0,u_0)$ to $(L_1,u_1)$.} It is well known that if $(L_0,u_0)$
and $(L_1,u_1)$ are isotopic framed links in $\sthree$, then there
is an ambient isotopy $\varphi_t:\sthree\to\sthree$ such that
$\varphi_1(L_0)=L_1$ and $d\varphi_1\circ u_0=u_1\circ
\varphi_1|_{L_0}$.

For a singular locally stable map $f:\sthree\to\rfour$ the bundle
$\Ker(df)$ is nowhere tangential to the singularity link
$\Sigma(f)$. Furthermore $\Ker(df)$ is trivial by Proposition
\ref{fredrik}. Hence $\Ker(df)$ admits a nowhere vanishing section
$u:\Sigma(f)\to\Ker(df)$. The pair $(\Sigma(f),u)$ is a framed
link in $\sthree$, whose isotopy class does not depend on the
choice of $u$.
\begin{definition}
\label{def_of_sigma}
Let $\sigma(f)$ be the isotopy class of the framed link
$(\Sigma(f),u)$.
\end{definition}

\begin{proposition}
\label{karlolov}
If $h_t:\sthree\to\rfour$ is an \lsh\ of singular locally stable
maps, then $\sigma(h_0)=\sigma(h_1)$.
\end{proposition}

\begin{proof}
Let $H:\sthree\times I\to \rfour\times I$ be the track of $h_t$,
and, as in the proof of Proposition \ref{fredrik}, let
$\Omega=J^{-1}(\Sigma^1)$ where $J:\sthree\times I\to
J^1(\sthree,\rfour)$ is the homotopy $J(p,t)=j^1h_t(p)$. Then the
kernel bundle $\Ker(dH)$ of the restriction of $dH$ to
$T(\sthree\times I)|\Omega$ is a trivial line bundle over $\Omega$
which is nowhere tangential to $\Omega$.

Let $q:\sthree\times I\to \sthree$ be the projection. Then
$q(\Omega\cap(\sthree\times \{t\}))=\Sigma(h_t)$, and $dq$ restricts to an
isomorphism from
$\Ker(dH)|(\Omega\cap(\sthree\times \{t\}))$ to $\Ker(dh_t)$. Let
$u:\Omega\to \Ker(dH)$ be a nowhere vanishing section. For each
$t\in I$ define $u_t:\Sigma(h_t)\to\Ker(dh_t)$ by
$u_t(p)=dq(u(p,t))$. Then, by \cite[Theorem 20.2]{AbRo},
$(\Sigma(h_t),u_t)$ is an isotopy from $(\Sigma(h_0),u_0)$ to
$(\Sigma(h_1),u_1)$. Hence $\sigma(h_0)=\sigma(h_1)$.
\end{proof}

\subsection{Orientations and metrics}
\label{orient}
The following conventions will be used throughout.
We equip $\R^n$ with its standard orientation
and metric. The closed ball in $\R^n$ of radius $r$
centered at the origin is denoted $D^n_r$. We orient $S^k$ as the
boundary of $D^{k+1}_1$, and we equip $S^k$ 
with the metric induced from $\R^{k+1}$.

\begin{remark}
We use the \emph{outward normal first convention} to orient the
boundary of an oriented manifold. That is, if $p$ belongs to the
boundary $\partial M$ of an oriented manifold $M$, then a basis
for $T_p\partial M$ is positively oriented if and only if an
outward pointing normal to $\partial M$ at $p$ followed by the
basis for $T_p\partial M$ is a positively oriented basis for
$T_pM$.
\end{remark}

\section{The invariants}
\label{invariants}
An \emph{\lsh\ invariant} is a locally constant function on the
space of locally stable maps. In this section we define two \lsh\
invariants: $\nu$, see Subsection \ref{nu}, and $\kappa$, see
Subsection \ref{ludde}. The invariant $\nu$ generalizes the normal
degree for immersions $\sthree\to\rfour$, and
$\kappa$ generalizes the $\R^5$-Smale invariant of immersions $S^3\to\R^4$,
i.e., the Smale invariant of an immersion $S^3\to\R^4$ composed with the
inclusion $\R^4\subset\R^5$, see \cite{EkSz}.
In Subsection \ref{conn_sum} we describe the behavior
of $\nu$ and $\kappa$ under taking connected sum of locally stable
maps, and in Subsection \ref{immersions} we relate $\nu$ and
$\kappa$ to the Smale invariant for immersions $\sthree\to\rfour$.

\subsection{The invariant $\nu$.}
\label{nu}
The normal degree $\nu$ defined on immersions $\sthree\to\rfour$
is a regular homotopy invariant. We extend $\nu$ to an \lsh\
invariant.

Let $f:\sthree\to\rfour$ be a singular locally stable map. Let
$\Gamma:\sthree\setminus\Sigma(f)\to\sthree\times\sthree$ be the
graph of the Gauss map of the immersion
$f|_{\sthree\setminus\Sigma(f)}$, and let
$q_1:\sthree\times\sthree\to\sthree$ be projection onto the first
factor. Then, by \cite[Lemma 5.1]{Gor}, the
intersection of the closure of the image of $\Gamma$ with
$q_1^{-1}(\Sigma(f))$ is a circle bundle $S$ over $\Sigma(f)$. The
bundle is such that its fiber over $p\in\Sigma(f)$ is
a great circle $S_p$ in $\{p\}\times S^3$. Denote the closure of
the image of $\Gamma$ by $M$. Then $M$ is a manifold with boundary
$S$.

Let $u:\Sigma(f)\to T\sthree|\Sigma(f)$ be a section with no
$u(p)$ in $\Ker(df)\oplus T\Sigma(f)$. (The section exists
by Proposition \ref{fredrik}.) Define
$U:\Sigma(f)\to\sthree\times\sthree$ by
\begin{equation*}
U(p)=\left(p,\frac{df(u(p))}{|df(u(p))|}\right).
\end{equation*}
Then, for each $p\in
\Sigma(f)$ there is a unique equator 2-sphere in $\{p\}\times S^3$
which contains $S_p$ and $U(p)$. Let $D_p$ be the
closed hemisphere of this 2-sphere that contains $U(p)$ and has $S_p$ as its boundary. 
The hemispheres $D_p$ form fibers in a disk bundle
$D$ over $\Sigma(f)$.

Let $q_2:\sthree\times\sthree\to\sthree$ be projection onto the
second factor. Equip $M$ with the orientation induced from
$\sthree\setminus\Sigma(f)$ via $\Gamma$. The orientation on $M$
defines an orientation on $S$. Equip the total space of $D$ with
the orientation that induces the opposite orientation on $S$. 
($D$ is orientable since $S$ is orientable.) Let
$m_f=q_2(M)$ and $d_f=q_2(D)$. Then $m_f+d_f$ is a $3$-cycle in
$\sthree$ whose homology class we denote by $\hv m_f+d_f\hh$.

\begin{definition}
\label{def_of_nu}
Let $\nu(f)$ be the integer defined by $\hv m_f+d_f\hh
=\nu(f)\hv\sthree\hh$, where $\hv\sthree\hh$ is the fundamental
homology class in $H_3(\sthree;\Z)$.
\end{definition}

\begin{remark} The integer $\nu(f)$ is independent of the choice of $u$.
Indeed, if $D'$ is the disk bundle over $\Sigma(f)$ obtained from
a different choice of section $u'$ of $T\sthree|\Sigma(f)$, and if
$d_f'=q_2(D')$, then $\hv m_f+d_f\hh =\hv  m_f+d'_f\hh $. The
equality is obvious if $u'$ is homotopic to $u$ through sections
of $T\sthree|\Sigma(f)$ with images in the complement of
$\Ker(df)\oplus T\Sigma(f)$. On the other hand, if $u'=-u$ and $p\in\Sigma(f)$, then
$D_p$ and $D'_p$ together form an equator 2-sphere which is the boundary of a 
hemisphere $H_p$ of $\{p\}\times\sthree$. This can be globalized by Proposition \ref{sture} and we get that 
$d_f-d'_f=\partial q_2(H)$ for a disk bundle $H$ over $\Sigma(f)$. 
Thus $d_f-d'_f$ is a zero-homologous cycle in $\sthree$.
\end{remark}

\begin{proposition}
\label{knocke}
If $h_t:\sthree\to\rfour$ is an \lsh\ of singular locally stable
maps, then $\nu(h_0)=\nu(h_1)$.
\end{proposition}

\begin{proof}
Define $J:\sthree\times I\to J^1(\sthree,\rfour)$ by
$J(p,t)=j^1h_t(p)$, and let $\Omega=J^{-1}(\Sigma^1)$. Let
$H:\sthree\times I\to\rfour\times I$ be the track of $h_t$, and
let $\Ker(dH)$ be the kernel bundle of $dH$ over $\Omega$. Define
$\Gamma:(\sthree\times I)\setminus\Omega\to \sthree\times
I\times\sthree$ by $\Gamma(p,t)=(p,t,n_t(p))$, where
$n_t:\sthree\setminus\Sigma(h_t)\to\sthree$ is the Gauss map of
$h_t|_{\sthree\setminus\Sigma(h_t)}$. Then, as above, the closure
of the image of $\Gamma$ is an oriented manifold $\mathcal{M}$
with oriented boundary a circle bundle $\mathcal{S}$ over
$\Omega$. Fix a section $u:\Omega\to T(\sthree\times I)|\Omega$
with no $u(p,t)$ in $T\Omega\oplus\Ker(dH)$, and, in analogy with
the above, use $u$ to define a disk bundle $\mathcal{D}$ over $\Omega$. Orient
the total space of $\mathcal{D}$ so that $\mathcal{D}$ induces the
opposite orientation on $\mathcal{S}$. Let $q:\sthree\times
I\times\sthree\to\sthree$ be projection onto the last factor. Then
$m_H=q(\mathcal{M})$ and $d_H=q(\mathcal{D})$ are 4-chains in
$\sthree$, and
$\partial(m_H+d_H)=m_{h_1}-m_{h_0}+d_{h_1}-d_{h_0}$. Hence $\hv
m_{h_1}+d_{h_1}\hh =\hv m_{h_0}+d_{h_0}\hh $.
\end{proof}

\subsection{The invariant $\kappa$.}
\label{ludde}
Let $D$ be a 4-dimensional disk in $\rfive_+$, the upper half
space of $\rfive$, with boundary $\partial D=D\cap\rfour=S^3$ and
such that $\sthree\times [0,\epsilon)$ is a collar neighborhood of
$\sthree$ in $D$ for some $\epsilon>0$. Orient $D$ so that the
induced orientation of $\partial D$ agrees with the fixed
orientation on $\sthree$. Let $f:\sthree\to\rfour$ be a locally
stable map, and let $i:\rfour\to\rfive$ be the standard inclusion.
We call an extension $F:D\to\rfive$ of $i\circ f$
\emph{admissible} if $F$ is locally stable and $F(p,t)=(i(f(p)),t)$ for
$(p,t)\in\sthree\times [0,\epsilon)$. The extension $F$ has
singularities of types $\Sigma^{1,0}$ and $\Sigma^{1,1}$. We call
the finitely many singularities of type $\Sigma^{1,1}$ \emph{cusps}. 
Ekholm and Sz\H{u}cs \cite{EkSz} proved that each
cusp of $F$ is naturally oriented. Thus each cusp is equipped with
a sign.

\begin{definition}
\label{def_of_kappa}
Let $\kappa(f)$ be the algebraic number of cusps of $F$.
\end{definition}

\begin{remark}
The orientation of a cusp $q$ of $F$ is defined as follows. The
total space of the cokernel bundle $\xi(F)$ of $dF$ over the
singularity surface $\Sigma^1(F)$ of $F$ is oriented, see
\cite{EkSz}. Choose a metric for $\Ker(dF)$ and let $\pm v(p)$ be
the two unit vectors in $\Ker(dF)$ at $p\in\Sigma^1(F)$. Then
$s:\Sigma^1(F)\to\xi(F)$ defined by $s(p)=d^2F(\pm v(p))$, where
$d^2F$ is the quadratic differential of $F$, is a section of
$\xi(F)$ that vanishes exactly at the cusps of $F$. The orientation
of $q$ is the local intersection number at $q$ of $s$ with the zero section.
\end{remark}

\begin{lemma}
\label{jonte}
Let $N$ be an oriented 4-sphere, and let $F:N\to\rfive$ be a
locally stable map. Then the algebraic number of cusps of $F$ is
zero.
\end{lemma}

\begin{proof}
The lemma is a special case of \cite[Lemma 3]{Szu}.
Here we only sketch the proof.
The map $F$ is null cobordant in the category of maps with at 
most $\Sigma^{1,1}$ singularities. That
is, there is a 5-manifold $W$ with boundary $N$ and a locally
stable map $B:W\to\rfive\times I$ whose restriction to $N$ equals
$F$ and the singularities of $B$ are of types $\Sigma^{1,0}$ and
$\Sigma^{1,1}$. The manifold $\Sigma^{1,1}(B)$ is an oriented
cobordism from the negative cusps of $F$ to the positive cusps of
$F$. Consequently the algebraic number of cusps of $F$ is zero.
\end{proof}

\begin{proposition}
\label{jojje}
The integer $\kappa(f)$ is independent of the choice of $D$, the
admissible extension $F$, and is an \lsh\ invariant of $f$.
\end{proposition}

\begin{proof}
This is a special case of \cite[Remark 3.2]{EkSz}. For completeness 
we include the proof.

Suppose $F:D_1\to\rfive$ and $G:D_2\to\rfive$ are admissible
extensions of $i\circ f$. Denote the algebraic number of cusps of
$F$ and $G$ by $\sharp\Sigma^{1,1}(F)$ and $\sharp\Sigma^{1,1}(G)$
respectively. Let $\rho:\rfive\to\rfive$ be reflection in
$\rfour$, let $\bar D_2$ be the manifold $\rho(D_2)$ with
orientation inducing the opposite orientation on $\partial\bar
D_2=\sthree$, and let $\bar G:\bar D_2\to\rfive$ be the
composition $\rho\circ G\circ\rho|_{\bar D_2}$. The algebraic
number of cusps of $F\cup \bar G: D_1\cup \bar D_2\to\rfive$
equals $\sharp\Sigma^{1,1}(F)-\sharp\Sigma^{1,1}(G)$. Thus
$\sharp\Sigma^{1,1}(F)=\sharp\Sigma^{1,1}(G)$ by Lemma
\ref{jonte}, which proves the first and second assertion. The
third assertion follows immediately.
\end{proof}

For a singular locally stable map $f:\sthree\to\rfour$ let
$\sharp\slk\Sigma(f)$ be \emph{the total self linking number of
$\Sigma(f)$.} That is, if $K_1,\dots,K_d$ are the components of
$\Sigma(f)$, and if $u:\Sigma(f)\to \Ker(df)$ is a nowhere
vanishing section, then let
\begin{equation*}
\sharp\slk\Sigma(f)=\sum_{j=1}^d\slk(K_j,u|_{K_j})
\end{equation*}
where $\slk(K_j,u|_{K_j})$ is the self linking number of the
framed knot $(K_j,u|_{K_j})$.

\begin{proposition}
\label{slk}
If $f:\sthree\to\rfour$ is a singular stable map, then $\kappa(f)$
and $\sharp\slk\Sigma(f)$ have the same parity.
\end{proposition}

Recall that a \emph{$k$-parameter deformation} of a map
$g:\sthree\to\rfour$ is a smooth family of maps
$g_\lambda:\sthree\to\rfour$, parameterized by the elements
$\lambda$ in an open neighborhood of $0\in\R^k$, such that
$g_0=g$. The deformation is called \emph{versal} if every
deformation of $g$ is equivalent, up to left-right action of
diffeomorphisms, to a deformation induced from $g_\lambda$.

\begin{proof}[Proof of proposition \ref{slk}]
The complement $\Delta$ in $C^\infty(\sthree,\rfour)$ of the
subspace of locally stable maps is called \emph{the discriminant}. The
discriminant is stratified,
$\Delta=\Delta^1\cup\Delta^2\cup\dots\cup\Delta^\infty$,
where each $\Delta^k$, for $k<\infty$, is a submanifold of
codimension $k$ in $C^\infty(\sthree,\rfour)$, and $\Delta^k$ is
contained in the closure of $\Delta^1$ for every $k$.

Let $h_t$ be a generic homotopy from the standard embedding
$\sthree\to\rfour$ to $f$. Then the path in
$C^\infty(\sthree,\rfour)$ determined by $h_t$ passes the
discriminant only through its top stratum $\Delta^1$,
transversely. The finitely many $s\in I$ for which
$h_s\in\Delta^1$ we call \emph{the critical instances} for the
homotopy. When $t$ varies between two critical instances the maps
$h_t$ are locally stable, and the total self linking number
$\sharp\slk\Sigma(h_t)$, when regarded a function of $t$, is
constant. Here we assume that $\sharp\slk\Sigma(h_t)=0$ if $h_t$
is an immersion.

Define the jump $J(s)$ in the total self linking number of
$h_t$ at a critical instance $s$ to be the difference
$J(s)=\sharp\slk\Sigma(h_{s+\epsilon})-\sharp\slk
\Sigma(h_{s-\epsilon})$, where $\epsilon>0$ is so small that $s$
is the only critical instance for $h_t$ in the interval
$s-\epsilon\leq t\leq s+\epsilon$. If $s_1<s_2<\dots<s_r$
are the critical instances for $h_t$, then
\begin{equation*}
\sharp\slk\Sigma(f)=\sum_{i=1}^rJ(s_i).
\end{equation*}
We show that $\sum_{i=1}^r J(s_i)$ is congruent to $\kappa(f)$
modulo $2$.

Normal forms for the maps in $\Delta^1$ were computed
by Houston and Kirk \cite{HuKi}. Namely,
for $h\in\Delta^1$ there is one exceptional point
$q\in\sthree$, and local coordinates centered at $q$ and $h(q)$
such that $h$, when expressed in these coordinates, is given
by one of the following five equations:
\begin{align}
h(x,y,z) &= (x,y,z^2,z(z^2+x^2+y^2)),\label{a} \\
h(x,y,z) &= (x,y,z^2,z(z^2+x^2-y^2)),\label{b} \\
h(x,y,z) &= (x,y,z^2,z(z^2-x^2+y^2)),\label{c} \\
h(x,y,z) &= (x,y,z^2,z(z^2-x^2-y^2)),\label{d} \\
h(x,y,z) &= (x,y,xz+z^4,yz+z^3).\label{e}
\end{align}
For each $p\in\sthree\setminus\{q\}$ there are the usual local coordinates
centered at $p$ and $h(p)$ respectively, in which $h$ is given by
one of the following two equations
\begin{equation*}
h(x,y,z) = (x,y,z,0) \quad\text{or}\quad h(x,y,z) = (x,y,z^2,yz).
\end{equation*}

Assume $s$ is a critical instance for the homotopy $h_t$.
The deformation of $h_{s}$ determined by $h_t$ is
equivalent to a versal $1$-parameter deformation $v_\lambda$ of
$h_s$. Let $q\in\sthree$ be the exceptional point for $h_s$.
Then $v_\lambda$ can be assumed constant in $\lambda$ outside small
coordinate neighborhoods centered at $q$ and $h_s(q)$ respectively, and,
when expressed in these coordinates,
to be given by expression \eqref{i}, \dots ,\eqref{v} provided that
$q$ is an exceptional point as in \eqref{a}, \dots ,\eqref{e} respectively.
Note that $J(s)$, or $-J(s)$, equals the jump in the total self linking
number of the versal deformation at $\lambda=0$.
\begin{align}
v_\lambda(x,y,z) &= (x,y,z^2,z(z^2+x^2+y^2+\lambda)),\label{i} \\
v_\lambda(x,y,z) &= (x,y,z^2,z(z^2+x^2-y^2+\lambda)),\label{ii} \\
v_\lambda(x,y,z) &= (x,y,z^2,z(z^2-x^2+y^2+\lambda)),\label{iii} \\
v_\lambda(x,y,z) &= (x,y,z^2,z(z^2-x^2-y^2+\lambda)),\label{iv} \\
v_\lambda(x,y,z) &= (x,y,xz+\lambda z^2+z^4,yz+z^3),\label{v}
\end{align}
\begin{figure}[t]
\begin{center}
\includegraphics[width=0.9\textwidth]{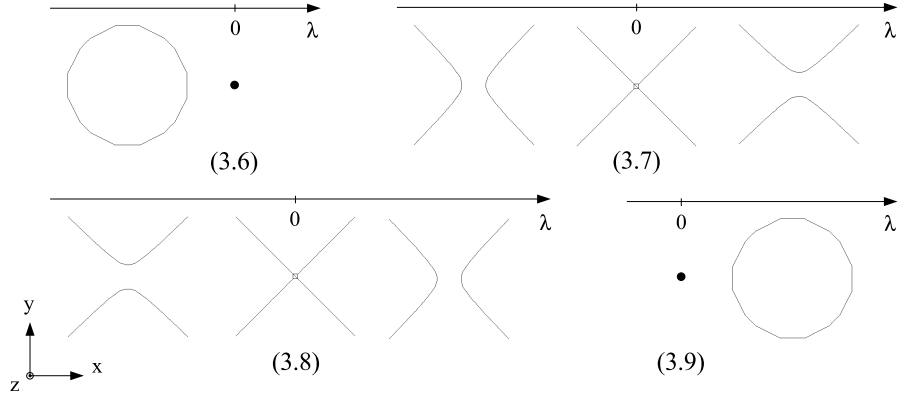}
\caption{Bifurcation diagrams for versal deformations (3.6)--(3.9).}
\label{fig1}
\end{center}
\end{figure}
\indent Figure \ref{fig1} contains bifurcation diagrams for the
singularity link under the versal deformations \eqref{i}--\eqref{iv}.
The kernel bundle for the differential $dv_\lambda$ is
everywhere parallel to the z-axis (that points towards the
observer). Deformations \eqref{i} and \eqref{iv} are responsible
for the vanishing and creation, and deformations \eqref{ii} and
\eqref{iii} for the fusion and splitting, of components of the
singularity link. None of these modifications of the
singularity link affects the total self linking number.

Figure \ref{fig2} contains a bifurcation diagram for the
singularity link under the versal deformation \eqref{v}. Outside
the box, decorated with the letter $n$, the kernel bundle of
$dv_\lambda$ is parallel to the z-axis. Inside the box the kernel
bundle is assumed to make $n$ twists about the displayed component
of the singularity link. Note that the fiber of the kernel bundle
of $dv_0$ over the exceptional point $q$ is tangential to the
singularity link. Thus $v_0$ has a cusp at $q$, that is, a
singularity of type $\Sigma^{1,1}$. As the diagram indicates, the
jump in the total self linking number of $v_\lambda$ at
$\lambda=0$ is $\pm 1$.

Let $H:\sthree\times I\to\rfour\times I$ be the track of the
homotopy $h_t$. Then $H$ is locally stable since $h_t$ is generic.
Moreover, by Proposition \ref{jojje}, $-\kappa(f)$ is the
algebraic number of cusps for $H$. Now $(q,s)\in\sthree\times I$
is a cusp for $H$ if and only if $s$ is a critical instance for
$h_t$ and $q$ is an exceptional point as in \eqref{e} for $h_{s}$. Thus,
if $H$ has $r$ cusps, at critical instances $s_1,s_2,\dots,s_r$,
then $J(s_i)=\pm 1$ for $i=1,\dots,r$, and $\kappa(f)\equiv
r\equiv\sum_{k=1}^rJ(s_i)$ modulo 2. Hence $\kappa(f)$ and
$\sharp\slk\Sigma(f)$ have the same parity.
\end{proof}
\begin{figure}[t]
\begin{center}
\includegraphics[width=0.9\textwidth]{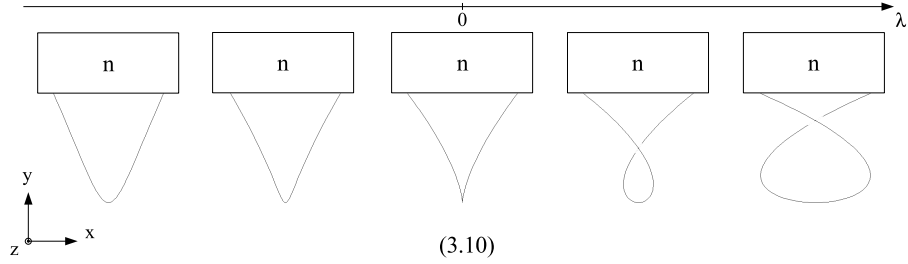}
\caption{Bifurcation diagram for versal deformation (3.10).}
\label{fig2}
\end{center}
\end{figure}
\begin{remark}
An alternative proof of Proposition \ref{slk} can be obtained from the 
proof of \cite[Proposition 8]{Juh2}.
\end{remark}
 
\subsection{The connected sum of locally stable maps.}
\label{conn_sum} 
Kervaire defined a connected sum operation
for immersions $S^k\to\R^n$, see \cite{Ker} for the explicit
construction. This operation can be analogously defined for
locally stable maps $S^k\to\R^n$, when $k<n$. We denote the
connected sum of two locally stable maps $f,g:S^k\to\R^n$ by
$f\natural g:S^k\to\R^n$. The connected sum $f\natural g$ is a locally stable
map which, since the singularity sets of $f$ and $g$ have at
least codimension 2 in $S^k$, is unique up to \lsh.

\begin{proposition}
If $f,g:\sthree\to\rfour$ are locally stable maps such that $f$ is
singular and $g$ is an immersion, then $f\natural g$ is singular
and $\sigma(f\natural g)=\sigma(f)$.
\end{proposition}

\begin{proof}
This is immediate from the definition of the connected sum.
\end{proof}

\begin{proposition}
\label{lsnu}
If $f,g:\sthree\to\rfour$ are locally stable maps, then
$\nu(f\natural g)=\nu(f)+\nu(g)-1$.
\end{proposition}
\begin{proof}
Milnor \cite{Mil} proved that for immersions $f,g:\sthree\to\rfour$
the normal degree of $f\natural g$ is $\nu(f\natural g)=\nu(f)+\nu(g)-1$.  A proof analogous
to the proof of Milnor's shows that the same formula holds for the
normal degree of a connected sum of two locally stable maps.
\end{proof}

\begin{proposition}
\label{lskappa}
If $f,g:\sthree\to\rfour$ are locally stable maps, then
$\kappa(f\natural g)=\kappa(f)+\kappa(g)$.
\end{proposition}

\begin{proof}
Let $p,q\in\sthree$ be such that $f(p)\ne g(p)$ and such that $f$
and $g$ are embeddings when restricted to small neighborhoods of
$p$ and $q$. Normalize $f$ at $f(p)$ and $g$ at $g(q)$, c.f.
\cite[Equations (2.1)]{Ker}.

Let $N$ be the 4-sphere in $\rfive$ obtained by gluing hemispheres
to the top and bottom of the cylinder $\sthree\times(-2,2)$.
Extend $i\circ f$ and $i\circ g$ to locally stable maps
$F,G:N\to\rfive$ such that $F(x,t)=(i(f(x)),t)$ and
$G(x,t)=(i(g(x)),t)$ for $(x,t)\in\sthree\times(-1,1)$. Then $F$
and $G$, when restricted to $N_+=N\cap\rfive_+$, are admissible
extensions of $i\circ f$ and $i\circ g$ respectively.

Let $H:\sfour\to\rfive$ be the connected sum of $F$ and $G$
obtained by ``joining the images $F(N)$ and $G(N)$ by a tube from
$F(p)$ to $G(q)$ with axis $F(p)+t(g(q)-f(p))$'', c.f. \cite[page 123]{Ker}.
Then $H$ is locally stable, and $H|_{N_+}:N_+\to\rfive$
is an admissible extension of $i\circ f\natural g$. The algebraic
number of cusps of $H|_{N_+}$ equals the sum of the algebraic
number of cusps of $F|_{N_+}$ and $G|_{N_+}$. Hence
$\kappa(f\natural g)=\kappa(f)+\kappa(g)$.
\end{proof}

\subsection{Immersions}
\label{immersions}
Smale \cite{Sma} classified immersions of spheres in Euclidean
spaces. Associated to each immersion $f:S^k\to\R^n$ is its Smale
invariant $\Omega_{n,k}(f)\in\pi_k(V_{n,k})$, where $V_{n,k}$ is
the Stiefel manifold of $k$-frames in $\R^n$, and two immersions
$f,g:S^k\to\R^{n}$ are regularly homotopic if and only if
$\Omega_{n,k}(f)=\Omega_{n,k}(g)$. Moreover, each element of
$\pi_k(V_{n,k})$ is the Smale invariant of an immersion
$S^k\to\R^n$. Thus the Smale invariant establish a one-one
correspondence between the set of regular homotopy classes of
immersions $S^k\to\R^n$ and $\pi_k(V_{n,k})$. Kervaire \cite{Ker}
proved that the set of regular homotopy classes of immersions
$S^k\to\R^n$ is an abelian group under connected sum, and that the
Smale invariant establish an isomorphism from this group to
$\pi_k(V_{n,k})$:

\begin{proposition}
If $f,g:S^k\to\R^n$ are immersions, then $\Omega_{n,k}(f\natural
g)=\Omega_{n,k}(f)+\Omega_{n,k}(g)$.
\end{proposition}

The group of immersions of $\sthree\to\rfour$ is isomorphic to
$\pi_3(V_{4,3})$. The space $V_{4,3}$ is homotopy equivalent to
$SO(4)$. A homotopy equivalence is obtained by adding a fourth
vector to each 3-frame to make it a positively oriented 4-frame,
and then apply the Gram-Schmidt process. We identify
$\pi_3(V_{4,3})$ with $\pi_3(SO(4))$ via the induced isomorphism.

The group $\pi_3(SO(4))$ is isomorphic to $\Z\oplus\Z$. To define
generators for $\pi_3(SO(4))$ we identify $\rfour$ with the
quaternions, $\sthree$ with the unit quaternions, $\rthree$ with
the pure quaternions, and define $\varrho:\sthree\to SO(3)$ by
$\varrho(x)y=x\cdot y\cdot x^{-1}$ for $x\in\sthree$, $y\in\rthree$. 
Then the homotopy classes
$[\sigma]$ and $[\rho]$ of $\sigma,\rho:\sthree\to SO(4)$ defined
by
\begin{equation*}
\sigma(x)y=x\cdot y \quad\text{and}\quad\rho(x)=\left[\begin{array}{cc} 1 & \\
                                       &
                                       \varrho(x)\end{array}\right]\quad (x\in\sthree,\,y\in\rfour)
\end{equation*}
respectively, generate $\pi_3(SO(4))$, see \cite[Chapter 8,
Proposition 12.9]{Hus}. Immersions $\sthree\to\rfour$ whose Smale
invariants equal $[\sigma]$ and $[\rho]$ were constructed by
Hughes \cite{Hug}.

\begin{lemma}
\label{immnormal}
If $f:\sthree\to\rfour$ is an immersion, and
$\Omega_{4,3}(f)=m[\sigma]+n[\rho]$, then $\nu(f)=m+1$.
\end{lemma}

\begin{proof}
For each $p\in\sthree$ let $Q(p)$ be the 3-frame $(p\cdot i,p\cdot
j,p\cdot k)$  translated to $T_p\sthree$. Further let $t(p)$ be
the 3-frame $df(Q(p))$ translated to the origin in $\rfour$, let
$n:\sthree\to\sthree$ be the Gauss map of $f$, and let $\Phi(p)$
be the 4-frame obtained by applying the Gram-Schmidt process to
$(n(p),t(p))$. Then, by \cite[Lemma 3.3.1]{Ekh2}, the
homotopy class of $\Phi:\sthree\to SO(4)$ equals $\Omega_{4,3}(f)+[\sigma]$.

Let $q:SO(4)\to\sthree$ be the fibration that sends each
matrix in $SO(4)$ to its first column vector. Then
$q_\ast[\sigma]$ generates $\pi_3(\sthree)$ and
\begin{equation*}
\nu(f)q_\ast[\sigma]=q_\ast[\Phi]=q_\ast(\Omega_{4,3}(f)+[\sigma])
=(m+1)q_\ast[\sigma]+nq_\ast[\rho]=(m+1)q_\ast[\sigma].
\end{equation*}
Accordingly $\nu(f)=m+1$.
\end{proof}

If $p:SO(5)\to V_{5,3}$ is the fibration that sends a matrix to
its last three column vectors, then
$p_\ast:\pi_3(SO(5))\to\pi_3(V_{5,3})$ is an isomorphism. We
identify $\pi_3(V_{5,3})$ with $\pi_3(SO(5))$ via $p_\ast$.
Moreover, if $j:SO(4)\to SO(5)$ is the inclusion
\begin{equation*}
j(A)=\left[\begin{array}{cc} 1 & \\ & A \end{array}\right],
\end{equation*}
then $j_\ast:\pi_3(SO(4))\to\pi_3(SO(5))$ is an epimorphism with
kernel generated by $2[\sigma]-[\rho]$, see \cite[Chapter 8,
Proposition 12.11]{Hus}.

Let $f:\sthree\to\rfour$ be an immersion and let
$i:\rfour\to\rfive$ be the standard inclusion.

\begin{lemma}
\label{omega}
$\Omega_{5,3}(i\circ f)=j_\ast(\Omega_{4,3}(f))$.
\end{lemma}

\begin{proof}
This is \cite[Lemma 3.3.3]{Ekh2}.
\end{proof}

\begin{lemma}
\label{immcusp}
$2\Omega_{5,3}(i\circ f)=\kappa(f)j_\ast[\sigma]$.
\end{lemma}

\begin{proof}
This is a consequence of \cite[Theorem 1(a)]{EkSz}.
\end{proof}

\begin{proposition}
\label{nu_kappa_det_imm}
Two immersions $f,g:\sthree\to\rfour$ are regularly homotopic if and
only if $\nu(f)=\nu(g)$ and $\kappa(f)=\kappa(g)$.
\end{proposition}

\begin{proof}
Lemmas \ref{immnormal}, \ref{omega}, and \ref{immcusp} imply that
$\Omega_{4,3}(f)=\Omega_{4,3}(g)$ if and only if $\nu(f)=\nu(g)$
and $\kappa(f)=\kappa(g)$.
\end{proof}

\section{Realization}
\label{realization}
In this section we prove that every framed link in $\sthree$ is the
framed singularity link of a locally stable map
into $\rfour$.

Let $V_3(T\sthree)$ be the bundle of 3-frames for $T\sthree$, and
let $Q:\sthree\to V_3(T\sthree)$ be an everywhere positively
oriented section. If $(u_1,u_2,u_3)$ is a 3-frame
at $p\in\sthree$, then we denote the matrix of coordinates for
$(u_1,u_2,u_3)$ when expressed in the basis $Q(p)$ for
$T_p\sthree$ by $[u_1,u_2,u_3]_Q$. Thus
$(u_1,u_2,u_3)=Q(p)[u_1,u_2,u_3]_Q$. Let $V_4(T\rfour)$ be the
bundle of 4-frames for $T\rfour$, and let $E:\rfour\to
V_4(T\rfour)$ be the section that takes each $q\in\rfour$ to the
standard basis for $T_q\rfour$. As above, if $(v_1,v_2,v_3,v_4)$
is a 4-frame at $q\in\rfour$, then we denote the
matrix of coordinates for $(v_1,v_2,v_3,v_4)$ when expressed in
the basis $E(q)$ for $T_q\rfour$ by $[v_1,v_2,v_3,v_4]_E$.

Let $V_3(T\rfour)$ be the bundle of 3-frames in
$T\rfour$, and let $GL^+(4)$ be the Lie group of real $4$-by-$4$
matrices with positive determinants. The map 
$\Psi:V_3(T\rfour)\to GL^+(4)$ defined by
$\Psi(v_1,v_2,v_3)=[v_1,v_2,v_3,v_1\times v_2\times v_3]_E$ is a
homotopy equivalence. Here, if $v_1,v_2,v_3\in T_q\rfour$, then
$v_1\times v_2\times v_3$ is the unique nonzero vector $w\in
T_q\rfour$ that is orthogonal to the linear span of $\{v_1,v_2,v_3\}$ and
satisfies $|w|^2=\det[v_1,v_2,v_3,w]_E$. Define
$\Xi:V_3(T\rfour)\to T\rfour$ by $\Xi(v_1,v_2,v_3)=v_1\times
v_2\times v_3$.

\begin{lemma}
\label{rel_h-prin}
Let $B$ be a proper tubular neighborhood of a link in $\sthree$. Let
$M=\sthree\setminus B$, and let $U$ be an open neighborhood of
$\partial M$ in $M$. Assume $g:U\to \rfour$ is an immersion, and
assume that the composition $\Psi\circ dg\circ Q|_U:U\to GL^+(4)$
extends to a continuous map $G:M\to GL^+(4)$. Then there exists an
immersion $f:M\to \rfour$ such that $G$ and $\Psi\circ df\circ Q|_M$
are homotopic relative $V$, where $V$ is an open neighborhood of
$\partial M$ contained in $U$. This, in turn, implies that $g$ and
$f|_U$ are regularly homotopic relative $V$. Accordingly
$g|_V=f|_V$, and hence $g|_V$ extends as an immersion to $M$.
\end{lemma}

\begin{proof} 
Combine \cite[Proposition 7.2.1]{ElMi} and \cite[Proposition 8.2.1]{ElMi}.
\end{proof}

\begin{proposition}
\label{aina}
Let $(L,u)$ be a framed link in $\sthree$. Then there exists a locally
stable map $f:\sthree\to\rfour$ such that $\Sigma(f)=L$ and such
that $u$ is a section in $\Ker(df)$.
\end{proposition}

\begin{proof}
Let $Q$, $E$, $\Psi$, and $\Xi$ be as defined in the introduction
to this section. Let $U$ be a proper tubular neighborhood of $L$ and let
$V$ be a proper tubular neighborhood of some knot in $\rfour$. Let
$L_0,\dots,L_d$ and $U_0,\dots,U_d$ be the components of $L$ and
$U$ respectively, with $L_j\subset U_j$, and let
$\alpha=\bigcup_{j=0}^{d}\alpha_j:\bigcup_{j=0}^{d}\sone\times\rtwo\to
\bigcup_{j=0}^{d}U_j=U$ and $\beta:\sone\times\rthree\to V$
be trivializations. On $U$ and $V$ we have the coordinate vector fields
$\partial_\theta,\partial_1,\partial_2$ and
$\partial_\theta,\partial_1,\partial_2,\partial_3$ induced by $\alpha$ and
$\beta$ respectively.
We assume that $\alpha$ and $\beta$ are orientation preserving
when $\sone\times\rtwo$ and $\sone\times\rthree$ are canonically
oriented, and we assume $\alpha$ to be such that
$\partial_2(p)=u(p)$ for each $p\in L$.

Define $f:U\to V$ by $f|_{U_j}=\beta\circ s\circ\alpha_j^{-1}$,
where $s:\sone\times\rtwo\to\sone\times\rthree$ is the map
$s(\theta,x,y)=(\theta,x,y^2,xy)$. Then $f$ has singularity link
$\Sigma(f)=L$ and, by the choice of $\alpha$, the frame $u$ of $L$
is a section in $\Ker(df)$.

Define $F:U\setminus L\to GL^+(4)$ by $F=\Psi\circ df\circ
Q|_{U\setminus L}$. For $j=0,\dots,d$ let $\mu_j:\sone\to U_j$
be the meridian $\mu_j(\varphi)=\alpha_j(0,\cos\varphi,\sin\varphi)$
of $L_j$. Then each loop $F\circ \mu_j$ is homotopically
nontrivial. To see this let $\gamma_j=f\circ\mu_j$ and define
$X:U\setminus L\to T\rfour$ by $X=\Xi\circ df\circ Q$.
Then $(F\circ\mu_j)(\varphi)=A_2(\varphi)A_3(\varphi)A_1(\varphi)$
where
\begin{align*}
A_1(\varphi) &= \left[\begin{array}{cc}
        	     [\partial_\theta(\mu_j(\varphi)),
                     \partial_1(\mu_j(\varphi)),\partial_2(\mu_j(\varphi))]_Q^{-1} &\\
			& 1\end{array}\right],\\
A_2(\varphi) &= \big[\partial_\theta(\gamma_j(\varphi)),\partial_1(\gamma_j(\varphi)), 
	\partial_2(\gamma_j(\varphi)),\partial_3(\gamma_j(\varphi))\big]_E,\\
A_3(\varphi) &= \left[\begin{array}{cccc}1 & 0           & 0         		& x_1(\varphi)\\
                                         0 & 1           & 0         		& x_2(\varphi)\\
                                         0 & 0           & 2\sin\varphi & x_3(\varphi)\\
                    			 0 & \sin\varphi & \cos\varphi  & x_4(\varphi)
                     \end{array}\right],
\end{align*}
and $x_1(\varphi),x_2(\varphi),x_3(\varphi),x_4(\varphi)$ are the
coordinates for $X(\mu_j(\varphi))$ when expressed in the basis
$(\partial_\theta(\gamma_j(\varphi)),\partial_1(\gamma_j(\varphi)),
\partial_2(\gamma_j(\varphi)),\partial_3(\gamma_j(\varphi)))$
for $T_{\gamma_j(\varphi)}\rfour$. We conclude, since the loops
$A_1$ and $A_2$ are\ nullhomotopic, that $F\circ\mu_j$ and $A_3$
are homotopic.

Let $Y:\sone\to T\rfour$ be the vector field
$Y(\varphi)=\cos\varphi\partial_2(\gamma_j(\varphi))-\sin\varphi\partial_3(\gamma_j(\varphi))$
along $\gamma_j$. Then $Y(\varphi)$ is normal to the hyperplane
$df(T_{\mu_j(\varphi)}\sthree)$ in $\rfour$, and $Y(\varphi)$ and
$X(\mu_j(\varphi))$ points into the same half space. The homotopy
$H_t(\varphi)=(1-t)X(\mu_j(\varphi))-tY(\varphi)$ induces a homotopy
from $A_3$ to the nontrivial loop $A_4:\sone\to GL^+(4)$ given by
\begin{equation*}
A_4(\varphi)=\left[\begin{array}{cccc}     1   &   0       &   0       &   0 \\
                        0   &   1       &   0       &   0 \\
                            0   &   0       &   2\sin\varphi     &   -\cos\varphi \\
                            0   &   \sin\varphi  &   \cos\varphi  &   \sin\varphi
\end{array}\right].
\end{equation*}
Hence $A_3$, and consequently $F\circ\mu_j$, are homotopically
nontrivial.

For $j=0,\dots,d$ let $\lambda_j:\sone\to U_j$ be the longitude
$\lambda_j(\theta)=\alpha_j(\theta,0,1)$. Further let
$r:\sone\times\rthree\to\sone\times\rthree$ be fiberwise rotation
$r(\theta,x,y,z)=(\theta,\cos\theta x+\sin\theta y,\cos\theta
y-\sin\theta x,z)$, and define $g:U\to V$ by $g|_{U_j}=\beta\circ
r\circ s\circ\alpha_j^{-1}$. Associated to $g$ is the map
$G:U\setminus L\to GL^+(4)$ defined by $G=\Psi\circ dg\circ
Q|_{U\setminus L}$. We prove that the loops $F\circ\lambda_j$ and
$G\circ\lambda_j$ are non-homotopic.

Let $\xi_j=f\circ\lambda_j$ and $\zeta_j=g\circ\lambda_j$. Define
$Z_j,W_j:\sone\to T\rfour$ by $Z_j=\Xi\circ df\circ
Q\circ\lambda_j$ and $W_j=\Xi\circ dg\circ Q\circ\lambda_j$. Then
the affine homotopies
$Z_t(\theta)=(1-t)Z_j(\theta)+t\partial_3(\xi_j(\theta))$ and
$W_t(\theta)=(1-t)W_j(\theta)+t\partial_3(\zeta_j(\theta))$ induce
homotopies between $F\circ\lambda_j$ and $B_1$, respectively
$G\circ\lambda_j$ and $B_2$, where $B_1,B_2:\sone\to GL^+(4)$ are
the loops
\begin{align*}
B_1(\theta) &= \big[df(Q_1(\lambda_j(\theta))),df(Q_2(\lambda_j(\theta))),
df(Q_3(\lambda_j(\theta))),\partial_3(\xi_j(\theta))\big]_E,\\
B_2(\theta) &=
\big[dg(Q_1(\lambda_j(\theta))),dg(Q_2(\lambda_j(\theta))),
dg(Q_3(\lambda_j(\theta))),\partial_3(\zeta_j(\theta))\big]_E,
\end{align*}
and $Q_1,Q_2,Q_3:\sthree\to T\sthree$ are the component vector
fields of $Q$. Define $C,D_1,D_2:\sone\to GL^+(4)$ by
\begin{align*}
C(\theta) &=\left[\begin{array}{cc}[\partial_\theta(\lambda_j(\theta)),
            \partial_1(\lambda_j(\theta)),\partial_2(\lambda_j(\theta))]_Q^{-1} & \\
                                        & 1\end{array}\right],\\
D_1(\theta) &= \big[\partial_\theta(\xi_j(\theta)),\partial_1(\xi_j(\theta)),
               \partial_2(\xi_j(\theta)),\partial_3(\xi_j(\theta))\big]_E,\\
D_2(\theta) &= \big[\partial_\theta(\zeta_j(\theta)),\partial_1(\zeta_j(\theta)),
               \partial_2(\zeta_j(\theta)),\partial_3(\zeta_j(\theta))\big]_E.
\end{align*}
Then $B_1(\theta)=D_1(\theta)P_1(\theta)C(\theta)$ and
$B_2(\theta)=D_2(\theta)P_2(\theta)C(\theta)$, where
$P_1,P_2:\sone\to GL^+(4)$ are given by
\begin{equation*}
P_1(\theta)=\left[\begin{array}{cccc}   1   &   0       &   0       &   0 \\
					0   &   1       &   0       &   0 \\
                          		0   &   0       &   2       &   0 \\
                            		0   &   1       &   0       &   1
\end{array}\right]
\quad\text{and}\quad
P_2(\theta)=\left[\begin{array}{cccc}     1   	&   0        &   0        &   0 \\
                        		\cos\theta    &   \cos\theta  &   2\sin\theta &   0 \\
                            		-\sin\theta   &   -\sin\theta &   2\cos\theta &   0 \\
                            		0             &   1           &   0           &   1
\end{array}\right]
\end{equation*}
respectively. The loops $P_1$ and $P_2$ are obviously non-homotopic. Since $D_1$
and $D_2$ are homotopic we conclude that $F\circ\lambda_j$ and
$G\circ\lambda_j$ are non-homotopic.

The group $H_1(\sthree\setminus L; \mathbb{Z}_2)$ is freely
generated by the homology classes of the meridians $\mu_j$. Each
longitude $\lambda_j$ is a cycle in $\sthree\setminus L$. If
$\lambda_j$ is homologous to a sum of an odd number of meridians
\emph{and} the loop $F\circ\lambda_j$ is nullhomotopic, then
replace $f|_{U_j}$ by $g|_{U_j}$ in the definition of $f$.
Conversely, if $\lambda_j$ is homologous to a sum of an even
number of meridians \emph{and} $F\circ\lambda_j$ is homotopically
nontrivial, then replace $f|_{U_j}$ by $g|_{U_j}$. The
redefined map $f:U\to V$ still has singularity link $\Sigma(f)=L$
and section $u$ in $\Ker(df)$.

Let $B=\bigcup_{j=0}^d\alpha_j(\sone\times D^2_1)$, and let $M=\sthree\setminus B$.
By Morse theory the pair $(M,\partial M)$ has a handle decomposition
$\partial M\subset M_0\subset M_1\subset M_2\subset M_3=M$ with
$M_0$ a closed collar neighborhood of $\partial M$ in $M$, $M_1$
connected since $M$ is connected, and $M_3$ obtained from $M_2$ by
the attachment of one 3-handle. The map $F_{-1}=F|_{\partial
M}:\partial M\to GL^+(4)$ is continuous. Our goal is to extend
$F_{-1}$ to a continuous map $F_3:M\to GL^+(4)$. Lemma
\ref{rel_h-prin} then assures that the restriction of
$f$ to a possibly smaller open neighborhood of $L$ than $U$
extends as an immersion to $M$.

Let $N_0,\dots,N_d$ be the components of $M_0$, and let
$\rho_0:M_0\to\partial M$ be a retraction that takes a 2-disk
$\Omega\subset\partial N_0$ to a point $b\in\partial M$. Let
$H_1,\dots, H_n$ be the 1-handles attached to $\partial M_0$ to
obtain $M_1$. Each handle $H_j$ has a core segment $I_j$, and we
may assume, possibly after a handle slide, that for $j=1,\dots,d$
the segment $I_j$ has one of its endpoints on $\partial N_0$ and
the other on $\partial N_j$, and for $j=d+1,\dots,n$ that
$I_j$ has both its endpoints in $\Omega$.

Let $F_0:M_0\to GL^+(4)$ be the continuous extension
$F_{-1}\circ\rho_0$ of $F_{-1}$. Extend $F_0$ arbitrarily
but continuously over the segments $I_1,\dots,I_d$. For
$j=d+1,\dots,n$ let $\lambda_j:\sone\to I_j\cup \Omega$ be a
simple loop passing through the segment $I_j$. The loop
$\lambda_j$ is a cycle in the first homology of $\sthree\setminus
L$ with $\mathbb{Z}_2$ coefficients. If $\lambda_j$ is homologous
to a sum of an odd number of meridians, then extend $F_0$ over
$I_j$ so that $F_0|_{I_j}:I_j\to GL^+(4)$ is a homotopically
nontrivial loop, based at $F_0(b)$. Otherwise extend $F_0$ over
$I_j$ as a constant map.

Let $\rho_1:M_1\to M_0\cup I_0\cup \dots\cup I_n$ be a retraction,
and let $F_1:M_1\to GL^+(4)$ be the extension 
$F_0\circ\rho_1$ of $F_0$. Consider the induced homomorphism
\begin{equation*}
(F_1)_\ast:H_1(M_1;\mathbb{Z}_2)\to \pi_1(GL^+(4))=\Z_2
\end{equation*}
where we have identified $\pi_1(GL^+(4))$ with
$H_1(GL^+(4);\mathbb{Z}_2)$ via the Hurewicz isomorphism. Now
$F_1$ extends continuously over $M_2$ if $(F_1)_\ast\hv v\hh=0$
for each loop $v:\sone\to M_1$ that parameterizes the boundary of
a core disk of a 2-handle in $M_2$. Here $\hv v\hh$ denotes the
homology class of $v$. So let $v$ be any such loop. Then, in
$H_1(M_1;\mathbb{Z}_2)$, for some $a_i,b_j\in\mathbb{Z}_2$,
\begin{equation*}
\hv
v\hh=\sum_{i=0}^da_i\hv\mu_i\hh+\sum_{j=0}^nb_j\hv\lambda_j\hh.
\end{equation*}
Moreover, if $\hv\lambda_j\hh=\sum_{i=0}^dc_{ij}\hv\mu_i\hh$ for
$j=1,\dots,n$, with $c_{ij}\in\mathbb{Z}_2$, when regarded as
homology classes in $H_1(\sthree\setminus L;\mathbb{Z}_2)$, then
by construction
\begin{equation*}
(F_1)_\ast\hv v\hh =
\sum_{i=0}^da_i(F_1)_\ast\hv\mu_i\hh+\sum_{j=0}^nb_j(F_1)_\ast\hv\lambda_j\hh
                         =\sum_{i=0}^d\big(a_i+\sum_{j=0}^nb_jc_{ij}\big).
\end{equation*}
Because the cycle $v$ is a boundary in $\sthree\setminus L$ we
have that
\begin{equation*} 0 = \sum_{i=0}^da_i\hv\mu_i\hh+\sum_{j=0}^nb_j\hv\lambda_j\hh
                = \sum_{i=0}^d\big(a_i+\sum_{j=0}^nb_jc_{ij}\big)\hv\mu_i\hh
\end{equation*}
in $H_1(\sthree\setminus L;\mathbb{Z}_2)$. Accordingly
$(F_1)_\ast\hv v\hh=0$, and we conclude that $F_1$ extends to a
continuous map $F_2:M_2\to GL^+(4)$. Finally $F_2$ extends to a
continuous map $F_3:M_3\to GL^+(4)$ because $\pi_2(GL^+(4))=0$.
Hence we have the desired extension of $F|_{\partial M}$ to $M$.
\end{proof}

\section{Sufficiency}
\label{sufficiency}
In this section we prove that if locally stable maps $f$ and $g$
have isotopic framed singularity links, then there exists an
immersion $k:\sthree\to\rfour$ such that $f$ is \lshc\ to
$g\natural k$. Throughout this section let $Q$, $E$, $\Psi$, and
$\Xi$ be as defined in the introduction to Section
\ref{realization}.

\begin{lemma}
\label{tomas}
Let $f,g:\sthree\to\rfour$ be singular stable maps such that
$\sigma(f)=\sigma(g)$. Then there exists a neighborhood $U$
of $\Sigma(g)$ and an \lsh\ $h_t:\sthree\to\rfour$ such that
$h_0=f$ and $h_1|_U=g|_U$.
\end{lemma}

\begin{proof}
For notational simplicity we assume that $\Sigma(f)$ and
$\Sigma(g)$ are knots. The proof below easily generalizes to the
case when $\Sigma(f)$ and $\Sigma(g)$ have more than one
component.

It follows from the assumption $\sigma(f)=\sigma(g)$, the Tubular
Neighborhood Theorem, Proposition \ref{sture}, and possibly after
replacing $f$ with a map to which $f$ is \lshc, that we can assume
that $\Sigma(f)=\Sigma(g)$, $f|_{\Sigma(f)}=g|_{\Sigma(g)}$, and
that there are tubular neighborhoods $U$ of
$\Sigma(f)=\Sigma(g)$ and $V$ of $f(\Sigma(f))=g(\Sigma(g))$, with
trivializations $\alpha^1,\alpha^2:\sone\times D^2_{2r}\to U$ and
$\beta^1,\beta^2:\sone\times D^3_s\to V$, such that $f(U)\subset
V$, $g(U)\subset V$, and such that $(\beta^1)^{-1}\circ f\circ
\alpha^1$ and $(\beta^2)^{-1}\circ g\circ \alpha^2$ satisfy
Equation \eqref{normform}. Furthermore we can assume that
$\alpha^1,\alpha^2$ and $\beta^1,\beta^2$ are orientation
preserving when $\sone\times D^2_{2r}$ and $\sone\times D^3_{s}$ are
canonically oriented.

Let $\Sigma=\Sigma(f)=\Sigma(g)$,
$\hat\Sigma=f(\Sigma(f))=g(\Sigma(g))$, and let
$\partial_\theta^i,\partial_1^i,\partial_2^i$ and
$\partial_\theta^i,\partial_1^i,\partial_2^i,\partial_3^i$ be the
coordinate vector fields on $U$ and $V$ induced by $\alpha^i$ and
$\beta^i$ for $i=1,2$. If
$(\partial_1^1|_\Sigma,\partial_2^1|_\Sigma)$ and
$(\partial_1^2|_\Sigma,\partial_2^2|_\Sigma)$, respectively
$(\partial_1^1|_{\hat\Sigma},\partial_2^1|_{\hat\Sigma},\partial_3^1|_{\hat\Sigma})$
and
$(\partial_1^2|_{\hat\Sigma},\partial_2^2|_{\hat\Sigma},\partial_3^2|_{\hat\Sigma})$,
are homotopic frames for $\Sigma$ and $\hat\Sigma$, then the
identity maps of $U$ and $V$ are ambient isotopic, say via
$\varphi_t:\sthree\to\sthree$ and $\psi_t:\rfour\to\rfour$, to
$\alpha^1\circ (\alpha^2)^{-1}:U\to U$ and $\beta^2\circ
(\beta^1)^{-1}:V\to V$ respectively. The homotopy
$h_t:\sthree\to\rfour$ defined by $h_t=\psi_t\circ
f\circ\varphi_t$ is an \lsh\ of $f$ such that $h_1|_U=g|_U$.

The frames $(\partial_1^1|_\Sigma,\partial_2^1|_\Sigma)$ and
$(\partial_1^2|_\Sigma,\partial_2^2|_\Sigma)$ for $\Sigma$ are
homotopic because $\sigma(f)=\sigma(g)$. Thus, after precomposing $f$ by the time-one map of an
ambient isotopy of $\sthree$ which restricts to an isotopy through
automorphisms of $U$ from the identity map to
$\alpha^1\circ(\alpha^2)^{-1}$, we can assume that
$\alpha^1=\alpha^2$.

Let $\alpha=\alpha^1=\alpha^2$. Define $\gamma:\sone\to\sthree$ by
$\gamma(\theta)=\alpha(\theta,0,r)$, and let $\mu:\sone\to\sthree$
be a parameterization of $\Sigma$ which is homotopic in $U$ to
$\gamma$. Define $\hat\gamma,\hat\mu:\sone\to\rfour$ by
$\hat\gamma=f\circ\gamma$ and $\hat\mu=f\circ\mu$. Let
$\partial_\theta,\partial_1,\partial_2$ be the coordinate vector
fields on $U$ induced by $\alpha$.
The frames
$(\partial_1^1|_{\hat\Sigma},\partial_2^1|_{\hat\Sigma},\partial_3^1|_{\hat\Sigma})$
and
$(\partial_1^2|_{\hat\Sigma},\partial_2^2|_{\hat\Sigma},\partial_3^2|_{\hat\Sigma})$
are homotopic if and only if the loops $A_1, B:\sone\to GL^+(4)$
defined by
\begin{align*}
A_1(\theta)&=[\partial_\theta^1(\hat\mu(\theta)),\partial_1^1(\hat\mu(\theta)),
	\partial_2^1(\hat\mu(\theta)),\partial_3^1(\hat\mu(\theta))]_E,\\
B(\theta)&=[\partial_\theta^2(\hat\mu(\theta)),\partial_1^2(\hat\mu(\theta)),
	\partial_2^2(\hat\mu(\theta)),\partial_3^2(\hat\mu(\theta))]_E.
\end{align*}
are homotopic. We prove that the homotopy class of $A_1$ is
determined by the homology class of $\gamma$ in
$H_1(\sthree\setminus\Sigma(f);\mathbb{Z}_2)$ and the homotopy
class of the frame
$(\partial_\theta|_\Sigma,\partial_1|_\Sigma,\partial_2|_\Sigma)$.
Then the same is true for $B$, and in particular $A_1$ and $B$ are
homotopic.

The loop $A_1$ is homotopic to $A_2:\sone\to GL^+(4)$ defined by
\begin{align*}
A_2(\theta)&= [\partial^1_\theta(\hat\gamma(\theta)),(\partial^1_1+r\partial^1_3)(\hat\gamma(\theta)),
              2r\partial^1_2(\hat\gamma(\theta)),\partial^1_3(\hat\gamma(\theta))]_E\\
           &= [df(\partial_\theta(\gamma(\theta))),df(\partial_1(\gamma(\theta))),
		df(\partial_2(\gamma(\theta))),\partial^1_3(f(\gamma(\theta)))]_E,
\end{align*}
which in turn is homotopic to the loop $A_3:\sone\to GL^+(4)$
given by
\begin{equation*}
A_3(\theta)=(\Psi\circ
df)(\partial_\theta(\gamma(\theta)),\partial_1(\gamma(\theta)),\partial_2(\gamma(\theta))).
\end{equation*}
Define $F:\sthree\setminus\Sigma(f)\to GL^+(4)$ by $F=\Psi\circ
df\circ Q$. Then the homotopy class of $F\circ\gamma$ depends
only on the homology class of $\gamma$ in
$H_1(\sthree\setminus\Sigma(f);\mathbb{Z}_2)$, see the proof of
Proposition \ref{aina}. Now $A_3(\theta)=(F\circ
\gamma)(\theta)A_4(\theta)$ where
\begin{equation*}
A_4(\theta)=\left[\begin{array}{cc}[\partial_\theta(\gamma(\theta)),
\partial_1(\gamma(\theta)),\partial_2(\gamma(\theta))]_Q & \\
                                 & 1 \end{array}\right].
\end{equation*}
Hence the homotopy class of $A_3$ is the sum of the homotopy
classes of $F\circ\gamma$ and $A_4$. Finally, a homotopy from
$\gamma$ to $\mu$ in $U$ induces a homotopy from $A_4$ to
$A_5:\sone\to GL^+(4)$ defined by
\begin{equation*}
A_5(\theta)=\left[\begin{array}{cc}[\partial_\theta(\mu(\theta)),
\partial_1(\mu(\theta)),\partial_2(\mu(\theta))]_Q & \\
                                 & 1 \end{array}\right].
\end{equation*}
Since the homotopy class of $A_5$ depends only on the homotopy
class of
$(\partial_\theta|_\Sigma,\partial_1|_\Sigma,\partial_2|_\Sigma)$
we conclude that the homotopy class of $A_1$ depends only on the
homology class of $\gamma$ and the homotopy class of
$(\partial_\theta|_\Sigma,\partial_1|_\Sigma,\partial_2|_\Sigma)$.
\end{proof}

\begin{lemma}
\label{relpara_h-prin}
Let $B$ be a proper tubular neighborhood of a link in $\sthree$. Let
$M=\sthree\setminus B$, and let $U$ be an open neighborhood of
$\partial M$ in $M$. Assume $f,g:M\to \rfour$ are
immersions, and assume that the restrictions $f|_U$ and $g|_U$ are
regularly homotopic via $h_t:U\to\rfour$. Then $f$ and $g$
are regularly homotopic via a homotopy that agrees with $h_t$ in
an open neighborhood of $\partial M$, possibly smaller than $U$,
if and only if the continuous map
\begin{equation*}
(\Psi\circ df\circ Q)\cup H\cup (\Psi\circ dg\circ Q): 
M\times\{0\}\cup U\times I\cup M\times\{1\}\to GL^+(4)
\end{equation*}
extends to a continuous map $M\times I\to GL^+(4)$. Here
$H:U\times I\to GL^+(4)$ is the homotopy $H(p,t)=(\Psi\circ
dh_t\circ Q)(p)$.
\end{lemma}

\begin{proof}
The lemma follows from \cite[Proposition 7.2.1]{ElMi} and \cite[Proposition 8.2.1]{ElMi}.
\end{proof}

\begin{proposition}
\label{f_lshc_gsumk}
Let $f,g:\sthree\to\rfour$ be stable maps such that
$\sigma(f)=\sigma(g)$. Then $f$ is \lshc\ to the connected sum of
$g$ with an immersion $k:\sthree\to\rfour$.
\end{proposition}

\begin{proof}
By Lemma \ref{tomas} we may assume that $f$ and $g$ have a common
singularity link $\Sigma=\Sigma(f)=\Sigma(g)$, and that
$f|_U=g|_U$ for an open neighborhood $U$ of $\Sigma$.

Let $W\subset U$ be a tubular neighborhood of $\Sigma$, let
$B$ be an open proper disk subbundle of $W$, let
$M=\sthree\setminus B$, and let $\partial M\subset M_0\subset
M_1\subset M_2\subset M_3=M$ be a handle decomposition of
$(M,\partial M)$ such that $M_0=W\setminus B$ and $M_3$ is
obtained from $M_2$ by the attachment of one 3-handle. Let
$h_t:U\to\rfour$ be the constant homotopy $h_t=f|_U=g|_U$. Define
$F ,G:\sthree\setminus\Sigma\to GL^+(4)$ by $F =\Psi\circ df\circ
Q$ and $G=\Psi\circ dg\circ Q$, define $H:(M_0\setminus\Sigma)\times
I\to GL^+(4)$ by $H(p,t)=(\Psi\circ dh_t\circ Q)(p)$, and define
$A_0:M\times\{0\}\cup M_0\times I\cup M\times \{1\}\to GL^+(4)$ by
$A_0=F \cup H\cup G$.

Assume, to begin with, that $A_0$ extends to a continuous map
\begin{equation*}
A_2:M\times\{0\}\cup M_2\times I\cup M\times \{1\}\to GL^+(4).
\end{equation*}
Let $D$ be the 3-handle attached to $M_2$ to obtain $M$, let $S$
be the topological 3-sphere $D\times\{0\}\cup\partial D\times
I\cup D\times \{1\}$, and consider the restriction $A_2|_S:S\to
GL^+(4)$. Suppose the homotopy class of
$A_2|_S$ equals $mi_\ast[\sigma]+ni_\ast[\rho]$, where $i:SO(4)\to
GL^+(4)$ is the inclusion and $[\sigma]$ and $[\rho]$ are the
generators for $\pi_3(SO(4))$ defined in Subsection
\ref{immersions}. Let $k:\sthree\to\rfour$ be an immersion with
Smale invariant $\Omega_{4,3}(k)=-mi_\ast[\sigma]-ni_\ast[\rho]$,
and let $g\natural k$ be a connected sum of $g$ and $k$ that
agrees with $g$ on $M_2$. Then, according to Kervaire \cite{Ker},
if we replace $g$ by $g\natural k$ (and let $A_2$ be defined by
$A_2(p,1)=(\Psi\circ d(g\natural k)\circ Q)(p)$ on
$D\times\{1\}$), the homotopy class of $A_2|_S$ is trivial, and
hence $A_2$ extends to a continuous map $M\times I\to GL^+(4)$.
This, in turn, implies that $f$ and $g\natural k$ are \lshc\ by
Lemma \ref{relpara_h-prin}. Thus, the proposition follows if we can
arrange it so that $A_0$ extends to $M\times\{0\}\cup M_2\times
I\cup M\times \{1\}$.

Let $W_0,\dots, W_d$ be the components of $W$. Assume $W$ is so
small that the sets $\tilde W_0,\dots,\tilde W_d$ defined by 
$\tilde W_j=f(W_j)=g(W_j)$ are pairwise disjoint.
Postcompose $f$ and $g$ by the time-one map of an ambient isotopy
of $\rfour$ that moves the images of $f$ and $g$ so that each
$\tilde W_j$ is contained in the closed ball in $\rfour$ with
radius $\frac{1}{2}$ centered at $(4j+\frac{5}{2},0,0,0)$, so that
$a_j'=(4j-1,0,0,0)\in\tilde W_{j-1}$ and $b_j'=(4j+2,0,0,0)\in\tilde
W_{j}$ for $j=1,\dots,d$, and so that small open subarcs of the first coordinate
axis in $\rfour$ centered at the $a'_j$:s and $b'_j$:s are
contained in $f(U)=g(U)$. Let $B_0,\dots,B_d$ be the
components of $B$, with $B_j\subset W_j$, let $V$ be a disk in the
topological boundary of $M_0$ in $M$, and let $b\in V$.

Let $H_1,\dots,H_n$ be the 1-handles attached to $\partial M_0$ to
obtain $M_1$. Each handle $H_j$ has a core segment $I_j$ with
endpoints $a_j$ and $b_j$ on the topological boundary of $M_0$ in
$M$. We may assume, possibly after a handle slide, that $a_j\in
W_{j-1}$, $b_j\in W_{j}$, $f(a_j)=g(a_j)=a_j'$, and
$f(b_j)=g(b_j)=b_j'$ for $j=1,\dots,d$, and that $a_j,b_j\in V$
for $j=d+1,\dots,n$. Moreover we may assume that the segments
$I_j$ are such that $f(I_1),\dots,f(I_{d})$ is a sequence of
pairwise disjoint embedded curves in $\rfour$, that each $f(I_j)$
intersects $\tilde W$ only at $a_j'$ and $b_j'$, and that $f(I_j)$
is contained in the first coordinate axis close to $a'_j$ and
$b'_j$. Now postcompose $f$ with the time-one map of an ambient
isotopy of $\rfour$ that keeps an open neighborhood of $f(W)$
fixed, and for $j=1,\dots,d$ moves the curve $f(I_j)$ onto
$\{(x,0,0,0):4j-1\leq x\leq 4j+2\}$.

For $j=1,\dots,d$ let $\gamma_j:I\to I_j$ be simple paths such
that $\gamma_j(0)=a_j$ and $\gamma_j(1)=b_j$, and define loops
$u_j:\sone\to I_j\times \{0\}\cup M_0\times I\cup I_j\times\{1\}$
by
\begin{equation*}
u_j(\theta) = \begin{cases}
              (a_j,\frac{2}{\pi}\theta)       & \text{if}\quad 0\leq\theta\leq \frac{\pi}{2},\\
              (\gamma_j(\frac{2}{\pi}\theta-1),1)     & \text{if}\quad\frac{\pi}{2}\leq\theta\leq \pi,\\
              (b_j,3-\frac{2}{\pi}\theta)         & \text{if}\quad\pi\leq\theta\leq \frac{3\pi}{2},\\
              (\gamma_j(4-\frac{2}{\pi}\theta),0)     & \text{if}\quad\frac{3\pi}{2}\leq\theta\leq 2\pi.
              \end{cases}
\end{equation*}
For $j=d+1,\dots,n$ let $\mu_j:\sone\to V\cup I_j$ be simple loops
such that $\mu_j(0)=b$, and define loops $u_j:\sone\to I_j\times
\{0\}\cup V\times I\cup I_j\times\{1\}$ by
\begin{equation*}
u_j(\theta) =  \begin{cases}
               (b,\frac{2}{\pi}\theta)     & \text{if}\quad 0\leq\theta\leq \frac{\pi}{2},\\
               (\mu_j(4\theta-2\pi),1)     & \text{if}\quad \frac{\pi}{2}\leq\theta\leq \pi,\\
               (b,3-\frac{2}{\pi}\theta)   & \text{if}\quad \pi\leq\theta\leq \frac{3\pi}{2},\\
               (\mu_j(8\pi-4\theta),0)     & \text{if}\quad \frac{3\pi}{2}\leq\theta\leq 2\pi.
               \end{cases}          
\end{equation*}

For $j=1,\dots,d$ let
$\lambda_j:\R\to\R$ be a nondecreasing smooth function such that
$\lambda_j(x)=0$ for $x<4j$ and $\lambda_j(x)=1$ for
$x>4j+1$. Define the isotopy $\psi^j_t:\rfour\to\rfour$
by
\begin{equation*}
\psi^j_t(x,y,z,w)=(x,y\cos(2\pi t\lambda(x))-z\sin(2\pi
t\lambda(x)),y\sin(2\pi t\lambda(x))+z\cos(2\pi t\lambda(x)),w).
\end{equation*}
Deform $Q$ so that for $j=1,\dots,d$ we have $F(a_j)=G(a_j)=G(b_j)=F (b_j)$, 
where $F $ and $G$ are as defined
above. Then each pair $F \circ \gamma_j$ and $G\circ\gamma_j$ is a
pair of loops in $GL^+(4)$. Inductively, starting with $j=d$ and
counting backwards, deform $f$ using $\psi_t^j\circ f$ into
$\psi_1^j\circ f$ if $F \circ \gamma_j$ and $G\circ\gamma_j$ are
\emph{not} homotopic. Otherwise leave $f$ unchanged. The resulting
map, which we also denote by $f$, is such that $F \circ \gamma_j$
and $G\circ\gamma_j$ \emph{are} homotopic for $j=1,\dots,d$.

Recall that $A_0:M\times\{0\}\cup M_0\times I\cup M\times \{1\}\to
GL^+(4)$ was defined as $A_0=F \cup H\cup G$. The map $A_0$
extends continuously to $M\times\{0\}\cup M_1\times I\cup M\times
\{1\}$ if and only if $(A_0)_\ast\hv u_j\hh=0$ for $j=1,\dots,n$.
Here $(A_0)_\ast$ is the homomorphism
\begin{equation*}
(A_0)_\ast:H_1(M\times\{0\}\cup M_0\times I\cup M\times
\{1\};\Z_2)\to \pi_1(GL^+(4)),
\end{equation*}
induced by $A_0$, where we have identified $\pi_1(GL^+(4))$ with $H_1(GL^+(4);\Z_2)$
via the Hurewicz isomorphism. Now $(A_0)_\ast$ satisfies
\begin{equation*}
(A_0)_\ast\hv u_j\hh=\begin{cases}
                     [F \circ\gamma_j]-[G\circ\gamma_j]   & \text{for}\quad j=1,\dots,d,\\
                     [F \circ\mu_j]-[G\circ\mu_j]     & \text{for}\quad j=d+1,\dots,n.
                     \end{cases}
\end{equation*}
But $[F \circ\gamma_j]=[G\circ\gamma_j]$ by construction, and the
homotopy classes $[F \circ\mu_j]$ and $[G\circ\mu_j]$ depend
only on the homology class $\hv \mu_j\hh$, which we proved in Proposition
\ref{aina}. Hence $[F \circ\mu_j]=[G\circ\mu_j]$, and $A_0$
extends to a continuous map
\begin{equation*}
A_1:M\times\{0\}\cup M_1\times I\cup M\times \{1\}\to GL^+(4),
\end{equation*}
Finally $A_1$ extends to $M\times\{0\}\cup M_2\times I\cup
M\times\{1\}$ because $\pi_2(GL^+(4))=0$.
\end{proof}

\section{Proofs of Theorems \ref{ett} and \ref{tva}}
\label{applications}
In this section we prove Theorems \ref{ett} and \ref{tva}.

\begin{proof}[Proof of Theorem \ref{ett}]
The implication from left to right follows from Propositions
\ref{karlolov}, \ref{knocke}, and \ref{jojje}. To prove the
opposite implication let $f,g:\sthree\to\rfour$ be locally stable
maps, and assume that $\sigma(f)=\sigma(g)$, $\nu(f)=\nu(g)$, and
$\kappa(f)=\kappa(g)$. By Proposition \ref{f_lshc_gsumk} there
exists an immersion $h:\sthree\to\rfour$ such that $f$ is \lshc\
to $g\natural h$. The equations
\begin{align*} 
&\nu(h) = \nu(g)+\nu(h)-1-\nu(g)+1=\nu(f)-\nu(g)+1=1\\
&\kappa(h) = \kappa(g)+\kappa(h)-\kappa(g)=\kappa(f)-\kappa(g)=0
\end{align*}
together with Proposition \ref{nu_kappa_det_imm} imply that $h$ is
regularly homotopic to the standard embedding $\sthree\to\rfour$.
Hence $f$ and $g$ are \lshc.
\end{proof}

\begin{proposition}
\label{immpar}
If $f:\sthree\to\rfour$ is an immersion, then the integers
$\nu(f)$ and $\frac{1}{2}\kappa(f)$ have different parity.
Conversely, for each pair $a,b\in\Z$ such that $a$ and $b$ have
different parity there exists an immersion $f:\sthree\to\rfour$
with $\nu(f)=a$ and $\kappa(f)=2b$.
\end{proposition}

\begin{proof}
The proposition follows from Lemmas \ref{immnormal}, \ref{omega},
and \ref{immcusp}.
\end{proof}

\begin{proof}[Proof of Theorem \ref{tva}]
By Proposition \ref{aina} there exists a locally stable map
$f:\sthree\to\rfour$ such that $\sigma(f)=\sigma$. Let
$g:\sthree\to \rfour$ be any locally stable map such that
$\sigma(g)=\sigma$. Then, by Proposition \ref{f_lshc_gsumk}, there
exists an immersion $h:\sthree\to\rfour$ such that $g$ is \lshc\
to $f\natural h$. Let $a=\nu(h)-1$ and
$b=\frac{1}{4}(\kappa(h)-2a)$. Then $b$ is an integer by
Proposition \ref{immpar}, and $\nu(g)=\nu(f)+a$ and
$\kappa(g)=\kappa(f)+2a+4b$ by Propositions \ref{lsnu} and
\ref{lskappa}. This proves the first and second assertion.

To prove the last assertion let $h:\sthree\to\rfour$ be an
immersion with Smale invariant $a[\sigma]+b[\rho]$. Then
$g=f\natural h$ is a locally stable map that satisfies
$\sigma(g)=\sigma$, $\nu(g)=\nu(f)+a$, and
$\kappa(g)=\kappa(f)+2a+4b$.
\end{proof}

\end{document}